\documentclass[11pt]{amsart}
\usepackage{amssymb}
\usepackage{amsmath}
\usepackage{tikz-cd}

\newtheorem{theorem}{Theorem}[section]
\newtheorem{question}{Question}
\newtheorem{lemma}[theorem]{Lemma}
\newtheorem{prop}[theorem]{Proposition}
\newtheorem{cor}[theorem]{Corollary}

\theoremstyle{definition}
\newtheorem{definition}[theorem]{Definition}

\newcommand{\A}{\mathcal{A}}
\newcommand{\C}{\mathcal{C}}
\newcommand{\HF}{\textnormal{HF}}

\newcommand{\asi}{\textnormal{asi}}
\newcommand{\dom}{\textnormal{dom}}
\newcommand{\res}{\upharpoonright}
\newcommand{\inv}{^{-1}}

\title{Descriptive Combinatorics, Computable Combinatorics, and ASI Algorithms}
\author{Long Qian, Felix Weilacher}
\email{}

\begin{document}

\maketitle

\begin{abstract}
    We introduce new types of local algorithms, which we call ``ASI Algorithms'', and use them to demonstrate a link between descriptive and computable combinatorics. This allows us to unify arguments from the two fields, and also sometimes to port arguments from one field to the other. As an example, we generalize a computable combinatorics result of Kierstead \cite{nearly_perfect}, and use it to get within one color of the Baire measurable analogue of Vizing's Theorem. We also improve Kierstead's result for multigraphs along the way. 
\end{abstract}

\section{Introduction}\label{sec:intro}

In this paper, we consider combinatorial problems on graphs, e.g. proper coloring, from an abstract point of view. Actually, we shall work on \textit{structured graphs} instead:

\begin{definition}\label{def:structured}
    A \textit{structured graph} is a structure in a finite language extending the usual language of graphs (in which there is a single binary relation symbol interpreted as adjacency) and whose restriction to that sub-language is a graph. I.e, such that the interpretation of the adjacency symbol is symmetric and irreflexive.
\end{definition}

Informally, all this means is that we expand the language of our graphs with extra predicates, so that we can encode information like orientations, multiple edges, and so on. 

For $G$ a structured graph, $V(G)$ will always denote the vertex set of $G$. $\rho_G$ will denote the path metric induced by $G$, for $X \subset V(G)$, $B_G(X,r)$ the radius $r$ ball around $X$ with respect to $\rho_G$. By default, all graphs in this paper are locally finite. 

\begin{definition}
A \textit{coloring problem} is a tuple $(\Pi,\Gamma)$, where $\Gamma$ is a class of structured graphs (e.g. bipartite, maximum degree $\Delta$, etc.) and $\Pi$ is a class of functions $V(G) \rightarrow \HF$ for $G \in \Gamma$ such that
\begin{enumerate}
    \item A structured graph $G$ is in $\Gamma$ if and only if all of its connected components are, and $\Gamma$ is closed under isomorphism.
    \item A function $c:V(G) \rightarrow \HF$ is in $\Pi$ if and only if its restriction to each connected component of $G$ is, and $\Pi$ is closed under pushforwards via isomorphisms.  
\end{enumerate}

We call elements of $\Pi$ \textit{$\Pi$-colorings}. Here and throughout, $\HF$ denotes the set of hereditarily finite sets.
\end{definition}

Typically, in the coloring problems we are interested in, $\Pi$ will be so called \textit{locally checkable} \cite{NS_lcl}. Somewhat informally, this means that membership of a coloring of in $\Pi$ is determined by its restriction to balls of a fixed finite diameter. For example, whether a coloring is proper can be determined by looking at all radius 1 balls. More formally, we have:

\begin{definition}\label{def:lcl}
    A coloring problem $(\Pi,\Gamma)$ is called \textit{locally checkable} if there is an $n \in \omega$ and a function $\mathcal{P}:\HF \rightarrow 2$ such that for any $G \in \Gamma$ and $c:V(G) \rightarrow \HF$, $c \in \Pi$ if and only if for each $x \in V(G)$, $\mathcal{P}$ returns 0 when applied to any rooted structured graph $\mathcal{H} \in \HF$ isomorphic to the rooted structured graph $G \res B_G(x,n)$ with root $x$.
    
    Moreover, it is called a \textit{computable} locally checkable labeling problem if $\mathcal{P}$ can be chosen to be computable.
\end{definition}

Definition \ref{def:lcl} is not important for much of this paper beyond the fact that it provides some cultural background on the types of combinatorial problems we are mostly interested in. It is, however, a good starting point for the notion of a \textit{local algorithm}. Again informally, a local algorithm is some procedure on a structured graph which ends in each vertex outputting some label which depends only on some finite radius ball around that vertex. Thus for example a locally checkable labelling problem amounts to a local algorithm where these output labels are always 1 or 0. 


In Section \ref{sec:algorithm} of this paper, we will formally define a new sort of local algorithm, which we call an \textit{ASI}-algorithm. 

Our motivation for doing this comes from two fields of what might be called ``definable combinatorics.'' In the first, ``descriptive combinatorics.'' We work with graphs on spaces which are nice from the point of view of descriptive set theory, e.g. standard Borel or Polish, and are interested in finding solutions to coloring problems which, as functions, are similarly nice, e.g. Borel, Lebesgue measurable, or Baire measurable. In the second, ``computable combinatorics''. Our graphs are on e.g. the natural numbers, and we are interested in finding solutions to coloring problems which, as functions, are e.g. computable. Some formal definitions are in order:

\begin{definition}
    Let $G$ be a structured graph, and $c:V(G) \rightarrow \HF$ a coloring of $G$.
    \begin{enumerate}
        \item $G$ is called \textit{Borel} if $V(G)$ is a Polish space, and all the relations/functions are Borel in/between the appropriate space/spaces.
        
        \item In the situation of (1), $c$ is called \textit{Borel} if it is Borel as a function, where $\HF$ is given the discrete topology.
        
        \item $G$ is called \textit{computable} if $V(G)$ is a computable subset of $\HF$, and all the relations and functions are computable. 
        
        \item In the situation of (3), $G$ is moreover called \textit{highly computable} if the function $\deg_G:V(G) \rightarrow \omega$ is computable.
        
        \item In the situation of (3), $c$ is called computable if it is computable as a function.
    \end{enumerate}
\end{definition}

The connection between local algorithms and descriptive combinatorics is an active area of current study; See \cite{Bernshteyn.distributed}, \cite{GR_grids}, and \cite{G+_trees}. A major theme in these works is that local algorithms can provide a unified setting in which to frame the combinatorial core of results in descriptive combinatorics. Our paper is an attempt to continue this tradition. In fact, our notion of an ASI algorithm seems closely related to the notion of a TOAST algorithm from \cite{GR_grids}. See Section \ref{subsec:toast} for further discussion. 

The connection between computable combinatorics and descriptive set theory does not seem well known, but as we hope to show in this paper, this is an avenue ripe for study. The project represented by this paper began with the observation that several results and techniques from descriptive combinatorics were analogous to those from computable combinatorics for highly computable graphs. For example, the bound on Baire measurable chromatic numbers due to Conley and Miller \cite{conley.miller.toast} matches that for computable chromatic numbers of highly computable graphs due to Schmerl \cite{schmerl_1980}, and is proved using similar ideas. Furthermore, the example found by the second author \cite{weilacher_2chi-1} showing the Baire measurable bound to be tight uses the exact same combinatorial trick as the example Schmerl uses to show his computable bound to be tight.

The language of ASI algorithms gives us a formal way of unifying these arguments. By Theorems \ref{thm:ASItoasi}, (together with work from other papers, see Theorem \ref{thm:asiresults}) and \ref{thm:ASItocomp}, if a coloring problem can be solved with an ASI algorithm, in can be solved in the highly computable setting, as well as in various descriptive settings. 

Though unification is valuable in its own right, we can also use ASI algorithms to transfer known results from one area into new results in another. For example, as a corollary of our Theorem \ref{thm:edge}, we get:

\begin{cor}\label{cor:edge}
Let $G$ be a multigraph with maximum edge multiplicity $p$. 
\begin{enumerate}
    \item If $G$ is a Borel multigraph on a Polish Space, $\chi_{BM}'(G) \leq \chi'(G) + p \leq \Delta(G) + 2p$
    \item If, in (1), $\mu$ is a Borel probability measure on $V(G)$ and $G$ is $\mu$-hyperfinite, the same bound holds for $\chi_\mu'$.
    \item If $G$ is a highly computable multigraph, the same bound holds.
\end{enumerate}
\end{cor}

Note that the second inequality $\chi'(G) + p \leq \Delta(G) + 2p$ is just the classical Vizing's theorem for multigraphs. For $p = 1$, (3) was shown by Kierstead \cite{nearly_perfect}. Recognizing Kierstead's proof as essentially an ASI algorithm allowed us to quickly obtain (1) and (2) in this case. Furthermore, the abstraction provided by the ASI algorithm setting made it more transparent how to adjust the proof for the case $p > 1$, allowing us to obtain the full versions of all three results. (The bound obtained in \cite{nearly_perfect} for $p > 1$ was worse by about $\sqrt{\chi'(G)}$ additional colors, and it was left open whether it could be improved.) 

In the descriptive setting (1) and (2) seem completely new. For simple graphs, the best existing bound in these settings seems to have been $\Delta + O((\log^3 \Delta) \sqrt{\Delta})$. This was pointed out to us by Anton Bernshteyn; It follows from a recent randomized distributed algorithm for edge coloring \cite{distributed_edge} and a theorem of his linking such algorithms and measurable colorings \cite{Bernshteyn.distributed}. It should be noted though that this bound holds in the measurable setting even without an assumption of hyperfiniteness. For multigraphs Matt Bowen has shown (personal communication) that $\chi_{BM}'(G) \leq \lceil \frac{3\Delta}{2} \rceil + 1$, and likewise for measure chromatic numbers with hyperfiniteness. This is close to best possible in general, but is improved by our bound when $\Delta$ is much greater than the maximum edge multiplicity.

We note that in both the descriptive and computable settings, whether a definable coloring using $\Delta + 1$ colors when $p = 1$ can always be found remains a major open question.

In Section \ref{sec:example}, we shall go over several such examples: Coloring problems for which known constructions from definable combinatorics can be re-cast as ASI algorithms. Many questions remain about which results from descriptive combinatorics can receive this treatment, and we list these and other further questions in Section \ref{sec:questions}. Also in Section \ref{sec:questions} are several quick non-examples (problems which we prove cannot be solved by our ASI algorithms) which provide partial answers to some of these questions.






\section{ASI Algorithms}\label{sec:algorithm}

\subsection{Basic Definitions}

\begin{definition}\label{def:algorithm}
Let $s \in \omega$. An \textit{ASI-$s$ algorithm} is a function $\A$ which on input $\mathcal{H} = (H,<,U_0,\ldots,U_s,C)$, where $H$ is a finite structured graph, $<$ is a linear order on $V(H)$, $C \subset V(H)$, and the $U_i$'s give a partition of $V(H)$, returns a coloring $C \rightarrow \HF$, and such that if $\phi:\mathcal{H} \rightarrow \mathcal{H'}$ is an isomorphism, $\A(\mathcal{H}) = \A(\mathcal{H'}) \circ (\phi \res C) $.
\end{definition}

Thus, the input to $\A$ is really the isomorphism type of the structure $\mathcal{H}$. The purpose of the linear order $<$ is to make $\mathcal{H}$ uniquely isomorphic to a canonical isomorphic structure in $\HF$. Thus it is performing a role similar to that of the unique vertex labels in the deterministic model of local algorithms. Note then that $\A$ can be construed as a real $\HF \rightarrow \HF$.

\begin{definition}
$\A$ as above is \textit{computable} if it is computable as a real $\HF \rightarrow \HF$.
\end{definition}

We now describe what it means to ``run'' such an algorithm on a graph:

\begin{definition}\label{def:application}
Let $s,r,n \in \omega$. Let $\mathcal{G} = (G,<,U_0,\ldots,U_s)$, where $G$ is a (possibly infinite) locally finite structured graph, $<$ is a linear order on $V(G)$, and the $U_i$'s give a partition of $V(G)$ which moreover has the property that for each $i$, $G^r \res U_i$ is component finite. Call such a partition an \textit{ASI-$s$ witness for $G$ with scale $r$}.

Let $\C := \bigsqcup_i \{C \mid C \textnormal{ is a connected component of } G^r \res U_i\}$, and let $\Tilde{G}$ be the graph on $\C$ in which components $C,D \in \C$ are adjacent if $\rho_G(C,D) \leq r$.

Let $\A$ be an ASI-$s$ algorithm. \textit{ $\A_{n}$ applied to $\mathcal{G}$}, denoted $\A_{n}[\mathcal{G}]$, is the coloring $V(G) \rightarrow \HF$ defined as follows: For each $C \in \C$, color $C$ according to $$\A(G \res W, < \res W, U_0 \cap W, \ldots, U_s \cap W, C),$$ where $$W = \bigcup \{D \in \C \mid \rho_{\Tilde{G}}(C,D) \leq n\}.$$ 
\end{definition}

For the above definition to make sense, $W$ must be finite. This is indeed the case, as $\Tilde{G}$ is locally finite since $G$ is locally finite and each element of $\C$ is finite.

Finally, we describe what it means for an ASI algorithm to ``solve'' a coloring problem.

\begin{definition}\label{def:solve}
Let $(\Pi,\Gamma)$ be a coloring problem, and $s,r,n \in \omega$. We say an ASI-$s$ algorithm $\A$ \textit{solves $(\Pi,\Gamma)$ in $n$ stages with scale $r$} if whenever $\mathcal{G} = (G,<,U_0,\ldots,U_s)$ is as in Definition \ref{def:application} with $G \in \Gamma$, $\A_n[\mathcal{G}] \in \Pi$.

If there exists such an $\A, s , r$, and $n$, we say $(\Pi,\Gamma)$ is in the class $\textnormal{ASI}(s)$. If moreover $\A$ is computable we say it is in the class $\textnormal{ASI}_c(s)$.
\end{definition}

By our definition of coloring problem, it is equivalent to only consider connected $G \in \Gamma$.

As an example/sanity check, let's discuss the case $s = 0$:

\begin{prop}
Let $(\Pi,\Gamma)$ be a coloring problem. The following are equivalent:
\begin{enumerate}
    \item Every finite $G \in \Gamma$ admits a $\Pi$-coloring.
    \item $(\Pi,\Gamma) \in \textnormal{ASI}(0)$.
\end{enumerate}
Moreover, if $\Pi$ is a computable locally checkable coloring problem, we can add:
\begin{enumerate}
    \setcounter{enumi}{2}
    \item ($\Pi,\Gamma) \in \textnormal{ASI}_c(0)$.
\end{enumerate}
\end{prop}

\begin{proof}
$(2) \Rightarrow (1)$: If $G \in \Gamma$ is finite, then for any linear order $<$ of $V(G)$, $(G,<,V(G))$ meets the hypotheses of Definition 2.4, so the algorithm witnessing (2) must produce a $\Pi$-coloring of $G$.

$(1) \Rightarrow (2)/(3)$: When $s = 0$, Definition \ref{def:solve} is vacuous for infinite connected $G$'s, and if $G \in \Gamma$ is finite and connected, we can just define $\A$ so that $\A(G,<,V(G),V(G))$ is always some $\Pi$-coloring of $G$. Moreover, if $\Pi$ is a computable locally checkable coloring problem, we can make $\A$ computable by computably enumerating the functions $V(G) \rightarrow \HF$, and stopping when we find one in $\Pi$.
\end{proof}

Evidently, the $s = 0$ case is somewhat degenerate, and we shall mostly ignore it for the remainder of the paper. We now describe the relationship between different $s > 0$:

\begin{prop}\label{prop:diff_asi}
If $s \leq s'$, $\textnormal{ASI}_{(c)}(s') \subset \textnormal{ASI}_{(c)}(s)$.
\end{prop}

\begin{proof}
Suppose the ASI-$s'$ algorithm $\A'$ solves $(\Pi,\Gamma)$ in $n$ stages with scale $r$. The following defines an ASI-$s$ algorithm $\A$ which solves $(\Pi,\Gamma)$ in $n$ stages with scale $r$:
$$\A(H,<,U_0,\ldots,U_s,C) = \A'(H,<,U_0,\ldots,U_s,\emptyset,\ldots,\emptyset,C).$$
This works because if $G$ is a graph and $U_0,\ldots,U_s$ an ASI-$s$ witness for $G$ with scale $r$, then $U_0',\ldots,U_{s'}'$ is an ASI-$s'$ witness for $G$ with scale $r$, where $U_i' = U_i$ for $i \leq s$ and $U_{i'} = \emptyset$ otherwise. 

Finally note that $\A$ is clearly computable if $\A'$ is.
\end{proof}

\subsection{From ASI Algorithms to Descriptive Combinatorics}

We now explain the connection between these algorithms and descriptive combinatorics. The name ``ASI'' stands for ``asymptotic separation index'', the following parameter introduced in \cite{dimension}:

\begin{definition}
    Let $G$ be a locally finite Borel graph. The \textit{asymptotic separation index of $G$}, denoted $\textnormal{asi}(G)$, is the least $s \in \omega$ such that for all $r \in \omega$, $G$ admits an ASI-$s$ witness $U_0,\ldots,U_s$ with scale $r$ where each $U_i$ is Borel, or $\infty$ if there is no such $s$.
\end{definition}

\begin{definition}
    A coloring problem $(\Pi,\Gamma)$ is in the class $\textnormal{asi}(s)$ if whenever $G \in \Gamma$ is Borel and has $\asi(G) \leq s$, it admits a Borel $\Pi$-coloring.
\end{definition}

\begin{theorem}\label{thm:ASItoasi}
For all $s \in \omega$, $\textnormal{ASI}(s) \subset \textnormal{asi}(s)$.
\end{theorem}

\begin{proof}
Suppose the ASI-$s$ algorithm $\A$ solves $(\Pi,\Gamma)$ in $n$ stages with scale $r$. Let $U_0,\ldots,U_s$ be a Borel ASI-$s$ witness for a Borel structured graph $G \in \Gamma$ with scale $r$. Let $<$ be a Borel linear order on $X := V(G)$. Then $\A[\mathcal{G}]$, where $\mathcal{G} = (G,<,U_0,\ldots,U_s)$, is a $\Pi$-coloring of $G$, so it suffices to show it is Borel.

For this, it suffices to show the map $X \rightarrow \HF \times \HF$ given by $x \mapsto (\mathcal{H},x')$, where $\mathcal{H}$ is the canonical $\HF$ representative of the isomorphism type of the structured graph $(G \res W, < \res W, U_0 \cap W, \ldots, U_s \cap W, C)$ from Definition \ref{def:application} where $C$ is the unique component in $\mathcal{C}$ containing $x$ and $x'$ is the image of $x$ under the isomorphism, is Borel. (Since $\A[\mathcal{G}]$ is the composition of this map with $(\mathcal{H},x') \mapsto \A(\mathcal{H})(x')$.) This is clear since $G,<$, and all the $U_i$'s are Borel.

\end{proof}

By Theorem \ref{thm:ASItoasi} and the following results, exhibiting an ASI algorithm to solve a problem has consequences in several definable settings:

\begin{theorem}[\cite{conley.miller.toast}, \cite{miller.2end}, \cite{dimension} ]\label{thm:asiresults}
Let $G$ be a locally finite Borel graph.
\begin{enumerate}
    \item There is a Borel $G$-invariant comeager set $C \subset V(G)$ such that $\asi(G \res C) \leq 1$.
    \item If $\mu$ is a Borel probability measure on $V(G)$, likewise for $\mu$-conull if $G$ is $\mu$-hyperfinite.
    \item If every connected component of $G$ has two ends, $\asi(G) = 1$.
    \item If $G$ is generated by the Borel action of a finitely generated group with polynomial growth rate, $\asi(G) \leq 1$.
\end{enumerate}
\end{theorem}

\subsection{From ASI Algorithms to Computable Combinatorics}

We now turn to the computable setting. To prove that ASI algorithms are useful in this setting, we need some analogue of Theorem \ref{thm:asiresults} telling us that we can at least sometimes effectively find ASI witnesses. In fact, the next lemma says that we always can, provided our graph is highly computable. The ideas in the construction have their origin in \cite{bean}, but were refined in \cite{schmerl_1980} and \cite{nearly_perfect}.

\begin{lemma}\label{lem:computable_asi} 
    Let $G$ be a highly computable graph and $r \in \omega$. Then we can find an ASI-1 witness $U_0,U_1 \subset V(G)$ with scale $r$ such that $U_0$ and $U_1$ are computable.
\end{lemma}

\begin{proof}
Since $G^r$ is still highly computable, it suffices to show this for $r = 1$. We can also assume WLOG that $V(G) \subset \omega$.

For $x \in V(G)$, let $f(x)$ be a $y \in V(G)$ minimizing $y + \rho_G(x,y)$, (with ties broken arbitrarily). Let $g(x) = f(x) + \rho_G(x,f(x))$. Note that $f$ and $g$ are computable: computing them requires only searching over vertices $y \leq x$ and paths starting at $x$ of length less than $x$. For $i \in 2$, let $U_i = \{x \in V(G) \mid g(x) \equiv i\  \textnormal{(mod } 2) \}$. These sets are computable since $g$ is, and they clearly partition $V(G)$.

It remains to show that each $G \res U_i$ is component finite. Note that if $(x,x') \in G$, $|g(x) - g(x')| \leq 1$. To see this, note that if $y = f(x)$, then $g(x') \leq y+\rho_G(x',y) \leq y + \rho(x,y) + 1 = g(x) + 1$. Therefore, since $g$ is integer valued, $g$ must be constant on any given $G \res U_0$ or $G \res U_1$ component, say $C$, say with value $c$. But then every vertex in $C$ lies within distance $c$ of $\{0,\ldots,c\} \subset V(G)$, and since $G$ is locally finite there are only finitely many such vertices.
\end{proof}

As promised, this gives us an analogue of Theorem \ref{thm:ASItoasi}:

\begin{definition}
    A coloring problem $(\Pi,\Gamma)$ is in the class $\textnormal{HComp}$ if whenever $G \in \Gamma$ is highly computable, it admits a computable $\Pi$-coloring.
\end{definition}

\begin{theorem}\label{thm:ASItocomp}
    $\textnormal{ASI}_c(1) \subset \textnormal{HComp}$.
\end{theorem}
\begin{proof}
Suppose the computable ASI-$1$ algorithm $\A$ solves $(\Pi,\Gamma)$ in $n$ stages with scale $r$, and let $G \in \Gamma$ be highly computable. By the lemma, let $U_0,U_1$ be a computable ASI-$1$ witness for $G$ with scale $r$. Let $<$ be a computable linear order on $X := V(G)$ (which clearly exists). Then $\A[\mathcal{G}]$, where $\mathcal{G} = (G,<,U_0,U_1)$, is a $\Pi$-coloring of $G$, so it suffices to show it is computable. This is clear as in the proof of Theorem \ref{thm:ASItoasi}.
\end{proof}

\section{Examples}\label{sec:example}

\subsection{An Alternate Viewpoint}

In this section, we give some examples of ASI algorithms which solve a few natural locally checkable coloring problems.

As is done in the world of distributed algorithms, we shall immediately abandon the formalism of Definitions \ref{def:algorithm} and \ref{def:application} in favor the following: To describe a computable ASI-$s$ algorithm $\A$ running for $n$ stages with scale $r$, we shall imagine we have a structure $\mathcal{G} = (G, <, U_0,\ldots,U_s)$ as in Definition \ref{def:application}, and imagine each $G^r \res U_i$-component (the elements of $\C$) as a sort of processor which has access to the restriction of $\mathcal{G}$ to it, and is capable of performing computations and sending messages to its neighboring components in $\Tilde{G}$. We shall thus describe $\A$ as a procedure which runs for $n$ stages, where in each stage, each of these components performs an arbitrary amount of computation, then sends messages to all of its neighbors. At the end of the procedure, each component needs to decide how to color itself.

To ease this process further, throughout this section, when we are in the midst of describing an ASI algorithm, we shall reserve all the variable symbols from the previous paragraph for the roles they had in that paragraph, not quantifying or declaring them except to specify additional hypotheses. 

\subsection{Vertex Coloring}

Let's start with the following easy example where no rounds of communication are needed. Let us write $\C_i$ for the set of $G^r \res U_i$-components, so that for example $\C = \bigsqcup_{i} \C_i$. Note that each $C \in \C$ has access to which $\C_i$ it comes from.

\begin{prop}
Let $k,s \in \omega$. Let $\Gamma$ be the class of $k$-colorable graphs, and $\Pi$ the problem of proper $k(s+1)$-coloring. Then $(\Pi,\Gamma) \in \textnormal{ASI}_c(s)$. In fact, there is a witnessing algorithm which runs for $0$ stages.
\end{prop}

\begin{proof}
We will use the scale $r = 1$. 
For each $i \leq s$, just have each $C \in \C_i$ pick some proper coloring for $G \res C$ using the colors $\{i k, i k + 1,\ldots, ik + (k-1)\}$. Such a coloring exists since $\chi(G \res C) \leq \chi(G) \leq k$, and of course we can effectively choose one. This algorithm works because if there is an edge between different $C,D \in \C$, then $C$ and $D$ use disjoint sets of colors.
\end{proof}

Note the hidden role of the linear order $<$ on the vertices in the above argument. What we really mean when we say ``effectively choose some $k$-coloring'', is that we have fixed some computable way of choosing, giving some $k$-colorable $(H,<) \in \HF$, with $<$ a linear order on $V(H)$, a $k$-coloring of $H$, (for example, just choose the least such coloring according to some computable well order of $\HF$.) and we are coloring $C$ accordingly, using that $(G \res C, < \res C)$ is uniquely isomorphic to a canonical such $H$. This is an important procedure that will be used in all our ASI algorithms. Henceforth, we will just say, e.g. ``choose a coloring'', without any mention of effectiveness or the above details.

With some communication, the number of colors needed can be reduced. The following is essentially from \cite{dimension}. That paper showed that the problem from the theorem is in the class $\textnormal{asi}(s)$. For $s = 1$, it was also shown in \cite{schmerl_1980} that the problem is in the class \textnormal{HComp}. Thus, by Theorems \ref{thm:ASItoasi} and \ref{thm:ASItocomp}, this theorem is a common generalization of both.

\begin{theorem}\label{thm:vertex}
Let $k,s$, and $\Gamma$ as in the previous Proposition, but now let $\Pi$ be the problem of proper $k(s+1) - s$-coloring. Then $(\Pi,\Gamma) \in \textnormal{ASI}_c(s)$.
\end{theorem}

\begin{proof}
We will use the scale $r = 4$. 
For each $i \leq s$ and $C \in \C_i$, have $C$ choose a proper $k$-coloring of $G \res B_G(C,1)$, call it $c_C^i$, say using the colors $\{0,1,\ldots,k-1\}$. Note by the choice of $r = 4$ that each $c^i := \bigcup_{C \in \C_i} c_C^i$ is a proper coloring. 

Now, following \cite{dimension}, define $d:V(G) \rightarrow \{0\} \sqcup (s+1) \times (k-1)$ by $d(x) = 0$ if $c^i(x) = k-1$ for each $i$ such that $x \in \dom(c^i)$, and $d(x) = (i,c^i(x))$ for $i$ least such that the above fails otherwise. Clearly our ASI algorithm can compute $d$, since each $C \in \C$ just needs to know the colorings $c_D^j$ for $D \in B_{\Tilde{G}}(C,1)$. It is easy to check that $d$ is a proper coloring, and it uses the promised number of colors.
\end{proof}

The following definition will be useful:

\begin{definition}\label{def:subordinate}
    Let $N,R \in \omega$. We say $E \subset V(G)$ is $(N,R)$-subordinate if each component of $G^R \res E$ is contained in the union of some radius $N$ ball in $\Tilde{G}$. (In particular, $G^R \res E$ is component finite).
\end{definition}

The proof of Theorem \ref{thm:vertex} illustrates how this is helpful for coloring problems: It used that for each $i$, the set $B(U_i,1)$ was $(1,1)$-subordinate. Thus, our ASI algorithm could make decisions locally about how to color $B(U_i,1)$-components (choosing the colorings $c_C^i$) without worrying about global coherence (hence $c^i$ was still a proper coloring). 

\subsection{Perfect Graphs}

We now turn to some examples from \cite{nearly_perfect}. The following is an analogue of Lemma 2.1 from \cite{BW.Konig}:

\begin{lemma}\label{lem:subordinate}
    Let $k,l_1,l_2 \in \omega$ with $r \geq 2 \cdot k \cdot \max(l_1,l_2) + 1$. In a constant number of stages, we can compute sets $S_i^j \subset V(G)$ for $j < s$ and $i < k$ such that, letting $S_i = \bigcup_j S_i^j$,
    \begin{enumerate}
        \item Each $B_G(S_i^j,l_1)$ is $(1,1)$-subordinate.
        \item For each $j$ and $i \neq i'$, $\rho_G(S_i^j,S_{i'}^j) \geq l_2$.
        \item Each $V(G) \setminus S_i$ is $(1,1)$-subordinate.
    \end{enumerate}
\end{lemma}

\begin{proof}
WLOG, $l_1 = l_2 =: l$. Each $S_i^j$ will be the set of points whose distance from $U_j$ is exactly $l \cdot i$. By choice of $r$ we can compute these sets after 1 stage, and property (2) is automatic.

For property (3), let $x_0,\ldots,x_m$ be a $G$-path in $V(G) \setminus S_i$. Suppose first that $x_0 \in B(C, l \cdot i)$ for some $C \in \C_j$ for some $j < s$. Then since this path avoids $S_i$, its distance from $U_j$ can never reach $l \cdot i$, so by choice of $r$ it must be contained in $B_G(C,l \cdot i)$, and hence the union of $B_{\Tilde{G}}(C,1)$. Otherwise, $x_0 \in U_s$ and has distance greater than $l \cdot i$ from each $U_j$ with $j < s$, and so clearly our path can never reach any $U_j$ with $j < s$, and so is contained in a component in $\C_s$.

Property (1) is similar: If $x_0,\ldots,x_m$ is a $G$-path in some $B_G(S_i^j,l)$ with $x_0 \in B_G(C,l \cdot (i+1))$ for $C \in \C_j$, then by choice of $r$ the whole path must in fact be contained in $B_G(C,l \cdot (i+1))$, and hence the union of $B_{\Tilde{G}}(C,1)$.
\end{proof}

Our two examples from \cite{nearly_perfect} will make similar use of this lemma. We start with a result about vertex colorings of perfect graphs. Recall that a graph is called \textit{perfect} if the chromatic number of each of its induced subgraphs is equal to the size of the largest clique contained in that subgraph.

\begin{theorem}\label{thm:perfect}
    Let $k,s \in \omega$. $\Gamma$ be the class of $k$-colorable perfect graphs and $\Pi$ the problem of proper $k+s$-coloring. Then $(\Pi,\Gamma) \in \textnormal{ASI}_c(s)$.
\end{theorem}

For $s = 1$, it was shown in \cite{nearly_perfect} that this problem is in $\textnormal{HComp}$, and our proof is directly based on the proof of that result. In the descriptive combinatorics setting, the result seems new, although the main result in \cite{BW.Konig} can be seen as a special case of it (since K\"{o}nig's theorem can be seen as saying that the edge graphs of bipartite graphs are perfect).

The following lemma gives the inductive step in our algorithm:

\begin{lemma}\label{lem:perfect_induction}
    Suppose $r \geq 4$. Let $k \in \omega$. 
    $G' \subset G$ a perfect subgraph with $\chi(G') \leq k+1$. Let $S^j \subset V(G)$ for each $j < s$ be such that each $B_{G}(S^j,1)$ and $V(G) \setminus S$ are $(1,1)$-subordinate, where $S := \bigcup_j S^j$.
    
    From this data, we can compute in a constant number of stages (but depending on $s$) $G'$-independent sets $A$ and $B^j$ for $j < s$ such that:
    \begin{enumerate}
        \item For each $j$, $B^j \subset B_G(S^j,1)$.
        \item Letting $G'' = G' \res V(G') \setminus (A \cup \bigcup_j B^j)$, $\chi(G'') \leq k$.
    \end{enumerate}
\end{lemma}

\begin{proof}
First, pick a proper $k+1$-coloring, say $c$, of $G' \res V(G) \setminus S$, say using the colors $\{0,\ldots,k\}$. As explained after Definition \ref{def:subordinate}, this can be done in a constant number of rounds by our ASI algorithm by the subordination hypothesis: Each $G \res (V(G) \setminus S)$-component can be colored separately and arbitrarily (note $G' \subset G$). Similarly, for each $j$, pick a proper $k+1$-coloring using the same colors, say $d^j$, of $G' \res B_G(S^j,1)$. Let $A = c\inv(\{0\})$, and $B^j = (d^j)\inv(\{0\})$ for each $j$. These sets are independent by definition of proper coloring, and (1) is clearly satisfied.

It remains to check (2). By the perfectness of $G'$, it suffices to show that every $k+1$-clique in $G'$ has at least one vertex in $A \cup \bigcup_j B^j$. Let $K \subset V(G)$ be such a clique. If $K$ meets some $S^j$, then it is contained in $B_G(S^j,1)$ since $G' \subset G$, so it was properly colored by $d^j$, but $d^j$ was a $(k+1)$-coloring, so some vertex of $K$ must have gotten the color 0 and thus be put in $B^j$. Else, $K$ is contained in $V(G) \setminus S$, and we can make the same argument with $c$ in place of $d^j$ to see some vertex of $K$ is in $A$.
\end{proof}

We can now prove the theorem:

\begin{proof}
We assume $G$ is perfect and $k$-colorable. We wish to apply Lemma \ref{lem:subordinate} with $k = k$, $l_1 = 1$, and $l_2 = 3$. Thus we use the scale $r$ required by the lemma, and we get sets $S_i^j$ as in the lemma. Let $G_0 = G$.

Let $i < k$, and suppose inductively that we have computed some induced subgraph $G_i = G \res V(G_i)$ with $\chi(G_i) \leq k - i$. By Lemma \ref{lem:perfect_induction}, in a constant number of additional rounds, we can compute independent sets $A_i,B_i^0,\ldots,B_i^{s-1} \subset V(G_i)$ such that $B_i^j \subset B_G(S_i^j,1)$ for each $j$ and, letting $G_{i+1} = G_i \res V(G_i) \setminus (A_i \cup \bigcup_j B_i^j)$, $\chi(G_{i+1}) \leq k - i - 1$. Thus the inductive hypothesis is maintained.

At the end, we have $\chi(G_k) = 0$, and so $V(G_k) = \emptyset$. I.e, the independent sets we removed along the way cover $V(G)$. Also, for each $j$, $B^j := \bigcup_i B_i^j$ is still independent by condition (2) from Lemma \ref{lem:subordinate}. Thus the $A_i$'s and $B^j$'s give a proper $k+s$-coloring of $G$.
\end{proof}

\subsection{Edge Colorings}

Our second example from \cite{nearly_perfect} concerns edge colorings. Recall that the following gives Corollary \ref{cor:edge}.

\begin{theorem}\label{thm:edge}
    Let $k,p,s \in \omega$. Let $\Gamma$ be the class of multigraphs of $k$-edge colorable multigraphs with maximum edge multiplicity at most $p$ and $\Pi$ the problem of proper $k+ps$-edge coloring. Then $(\Pi,\Gamma) \in \textnormal{ASI}_c(s)$.
\end{theorem}

As we mentioned in the introduction, while the $s = 1$, $p = 1$ version of this problem is shown to be in \textnormal{HComp} in \cite{nearly_perfect} (and the proof of that result is the basis of our proof), the bound in that paper for $p > 1$ (see Corollary 3.2) is worse than the one which follows from our result and Theorem \ref{thm:ASItocomp}. In the descriptive combinatorics setting, this bound seems to completely new.

We will follow the same outline as we did in the proof of Theorem \ref{thm:perfect}, but the finite combinatorial analysis will be more difficult. Our main tool here is the following form of the Vizing adjacency lemma for multigraphs (\cite{VAL}, Theorem 6):

\begin{theorem}\label{thm:VAL}
    Let $H$ be a multigraph with maximum degree at most $k \in \omega$, and $e$ an edge between vertices $x,y \in V(H)$. Letting $p_H(x',y')$ denote the number of edges in $H$ between vertices $x'$, and $y'$, suppose
    \begin{enumerate}
        \item $\chi'(H - \{e\}) \leq k$
        \item For every $z \in N_H(\{x\})$, $\deg_H(z) \leq k - p_H(x,z)+1$
        \item The number of $z \neq y$ for which equality holds in the previous condition is at most $k - \deg_H(y) - p_H(x,y) + 1$.
    \end{enumerate}
    Then also $\chi'(H) \leq k$.
\end{theorem}

We will use the following consequence of this theorem, which is an analog of Lemma 4.0 from \cite{nearly_perfect} for multigraphs:

\begin{lemma}\label{lem:weakVAL}
    Let $H$ be a multigraph with maximum degree at most $k \in \omega$, maximum edge multiplicity at most $p \in \omega$, and $U \subset V(H)$. Suppose 
    \begin{enumerate}
        \item $\chi'(H \res (V(H) \setminus U)) \leq k$.
        \item For every $y \in N_H(U)$, $\deg_H(y) \leq k - p$. 
    \end{enumerate}
    Then also $\chi'(H) \leq k$.
\end{lemma}

\begin{proof}
We add one edge from $H \setminus (H \res (V(H) \setminus U))$ at a time, using Theorem \ref{thm:VAL} each time to argue that the edge chromatic number stays below $k$. Indeed, if $e$ is an added edge and we apply the theorem with $x$ and endpoint of $e$ in $U$, then for all $z \in N_H(\{x\})$, the inequality in (2) of the theorem is strict by hypothesis (2) of this lemma. Thus (2) and (3) from the theorem hold, as desired.
\end{proof}

Before moving on, we pause to note that greedy algorithms can easily be carried out in our setting. The following formalizes this obvious fact in one specific instance which we will use:

\begin{lemma}\label{lem:greedy}
    Let $G' \subset G$ a multigraph and $M' \subset G'$ a matching, and suppose $r \geq 4$ (recall this is our scale).
    
    From this data, in a constant number of stages (but depending on $s$), we can compute a maximal matching $M \supset M'$ of $G'$.
\end{lemma}

\begin{proof}
    For each $j \leq s$, $B_G(U_j,1)$ is (1,1)-subordinate, so going one $j$ at a time, we can greedily add edges with both endpoints in $B_G(U_j,1)$ to our matching. Since any edge of $G$ is contained in some $B_G(U_j,1)$, the matching we end up with will be maximal.
\end{proof}

We can now give the inductive step in our algorithm. It is analogous to Lemma \ref{lem:perfect_induction}.

\begin{lemma}\label{lem:edge_induction}
    Let $k \in \omega$. Let $G' \subset G$ a multigraph with maximum edge multiplicity at most $p \in \omega$, and $\chi'(G') \leq k+1$. Let $S^j \subset V(G)$ for each $j < s$ be such that each $B_G(S^j,3)$ and $V(G) \setminus S$ are $(1,1)$-subordinate, where $S := \bigcup_j S^j$.
    
    From this data, we can compute in a constant number of stages (but depending on $s$) matchings $M$ and $N_i^j$ for $i < p$, $j < s$ such that:
    \begin{enumerate}
        \item For each $i,j$, each edge of $N_i^j$ is contained in $B_G(S^j,3)$.
        \item Letting $G'' = G \setminus (M \cup \bigcup_{i,j} N_i^j)$, $\chi'(G'') \leq k$.
    \end{enumerate}
\end{lemma}

\begin{proof}
As in the proof of Lemma \ref{lem:perfect_induction}, use the subordination hypothesis to pick a proper $k+1$-edge coloring $d^j : G' \res B_G(S^j,3) \rightarrow k+1$ for each $j$. For each $j$ and $i < p$, let $N_i^j = (d^j)\inv(\{i\})$. These are matchings by definition of proper edge coloring, and clearly condition (1) is satisfied.

Let $G^* = G \setminus \bigcup_{i,j} N_i^j$. By the subordination hypothesis once again, pick a proper $k+1$-edge coloring $c:G^* \res (V(G) \setminus S) \rightarrow k+1$. By Lemma \ref{lem:greedy}, we can compute a maximal matching, say $M$, of $G^*$ with $c\inv(\{0\}) \subset M$. Then $G'' = G^* \setminus M$.

It remains to check $\chi'(G'') \leq k$. First note $\Delta(G'') \leq k$: If $x \in B_G(S^j,2)$ for some $j$, then $B_{G'}(\{x\},1) \subset B_G(S^j,3)$, so $d^j$ gave a $k+1$-coloring of the $G'$-edges with $x$ as an endpoint. We removed $p$ color sets in the construction of $G^*$, though, so in fact
$$\deg_{G''}(x) \leq \deg_{G^*}(x) \leq k - p + 1 \leq k.$$
Else, $B_{G'}(\{x\},1) \subset V(G) \setminus S$, so the same argument with $c$ in place of $d^j$ works.

Let $W = \{x \in B_G(S,2) \mid \deg_{G''}(x) > k - p\}$. We wish to apply Lemma \ref{lem:weakVAL} with $H = G''$ and $U = T := W \cap B_G(S,1)$. For condition (2) from the lemma, it suffices to show $W$ is $G''$-independent. Suppose to the contrary $x,y \in W$ are adjacent. By the inequality in the previous paragraph, we must then have $\deg_{G^*}(x) = \deg_{G^*}(y) = k - p + 1$. In the construction of $G''$ from $G^*$, though, we removed a maximal matching, and so the degree of at least one of $x$ and $y$ must have dropped by one, contradicting the definition of $W$.

For condition (1), we need to see $\chi'(G'' \res (V(G) \setminus T)) \leq k$. Call this graph $H$. We will apply Lemma \ref{lem:weakVAL} to $H$ with $U = S \setminus W$. For condition (2) from the lemma, note that if $y \in N_H(S \setminus W)$, $y \in B_G(S,1) \setminus W$, so $\deg_H(y) \leq \deg_{G''}(y) \leq k - p$ by definition of $W$. Condition (1) of the lemma holds since $c$ was a $k+1$ coloring of $G^* \res (V(G) \setminus S)$ and we removed a color in building $G''$.
\end{proof}

We can now prove the theorem:

\begin{proof}
We assume $G$ is a multigraph of maximum edge multiplicity at most $p$ and $\chi'(G) \leq k$. We wish to apply Lemma \ref{lem:subordinate} with $k = k$, $l_1 = 3$, and $l_2 = 6$. Thus we use a scale $r$ big enough to meet the requirements of the lemma, and also with $r \geq 4$. Thus we get sets $S_i^j$ as in the lemma. Let $G_0 = G$.

Let $i < k$, and suppose inductively that we have computed some $G_i \subset G$ with $\chi'(G_i) \leq k - i$. By Lemma \ref{lem:edge_induction}, in a constant (depending on $s$) number of additional rounds, we can compute matchings $M_i$ and $N_{i,l}^j \subset G_i$ for $l < p$, $j < s$ so that each edge of each $N_{i,l}^j$ is contained in $B_G(S_i^j,3)$ and, letting $G_{i+1} = G_i \setminus (M_i \cup \bigcup_{l,j} N_{i,j}^l)$, $\chi'(G_{i+1}) \leq \chi'(G) - i - 1$. Thus the inductive hypothesis is maintained.

At the end, we have $\chi'(G_k) = 0$, and so $G_k$ is empty. I.e, the matchings we removed along the way union to all of $G$. Furthermore, for each fixed $j < s$ and $l < p$, $N_{l}^j := \bigcup_i N_{i,l}^j$ is still a matching by condition (1) of Lemma \ref{lem:edge_induction} and condition (2) of Lemma \ref{lem:subordinate}. Therefore the $M_i$'s and $N_l^j$'s give a proper $k + ps$-edge coloring of $G$.
\end{proof}

\section{Further questions}\label{sec:questions}

\subsection{Relationships Between Different ASI-classes} In Section \ref{sec:algorithm}, we defined several classes of coloring problems, and proved easy inclusions between many of them. It is natural to ask whether any of these inclusions is tight.

An open question from \cite{dimension} is whether there is any Borel graph $G$ with $1 < \asi(G) < \omega$. The following is the natural analogue of this question for ASI algorithms:

\begin{question}\label{q:intermediate}
    Are any of the inclusions in Proposition \ref{prop:diff_asi} strict for $s \geq 1$?
\end{question}

Though it is hard to imagine what a Borel graph with ``intermediate asi'' would look like, it does not seem too far-fetched that some combinatorial argument could prove some of the bounds from Section \ref{sec:example} for $s > 1$ to be tight. For example, a special case of Theorem \ref{thm:vertex} is that if $\Gamma$ is the class of acyclic graphs and $\Pi$ is the problem of proper $s+2$-coloring, then $(\Pi,\Gamma) \in \textnormal{ASI}_c(s)$. We can therefore ask if $s+1$ colors are enough when $s > 1$:

\begin{question}\label{q:tree_coloring}
    Let $s > 1$. With $\Gamma$ as in the previous paragraph and $\Pi$ the problem of proper $s+1$-coloring. Is $(\Pi,\Gamma) \in \textnormal{ASI}(s)$?
\end{question}

As suggested, a negative answer to this would totally answer Question \ref{q:intermediate}, showing that in fact all of the inclusions are strict. We remark that game theoretic methods have been very effective at finding negative answers for this coloring problem in related contexts. (This is why we have chosen it as a candidate for getting at Question \ref{q:intermediate}.) For example, see \cite{G+_trees} for the determinstic local algorithm context, or \cite{Marks} for the Borel context.

We can also ask whether there are any non-contrived differences between $\textnormal{ASI}$ and $\textnormal{ASI}_c$:

\begin{question}\label{q:computable}
    Let $s > 0$. Is there any computable locally checkable coloring problem which is in $\textnormal{ASI}(s)$ but not $\textnormal{ASI}_c(s)$?
\end{question}

The restriction to computable locally checkable coloring problems is necessary to avoid trivial counterexamples. For example, $\Gamma$ could be the class of graphs with vertex labels from the set of turing machines, and $\Pi$ could be the problem where each vertex must output whether its turing machine will halt. 

\subsection{How much can we unify?}

It is also natural to ask about the strictness of inclusions between our ASI classes and established classes of problems from definable combinatorics. Let us make the following definitions:

\begin{definition}\label{def:measurable}
    Let $(\Pi,\Gamma)$ be a coloring problem.
    \begin{enumerate}
        \item $(\Pi,\Gamma)$ is in the class Baire if whenever $G \in \Gamma$ is Borel, it admits a Baire measurable $\Pi$-coloring.
        \item $(\Pi,\Gamma)$ is in the class HypMeas if whenver $G \in \Gamma$ is Borel and $\mu$ is a Borel probability measure on $V(G)$ for which $G$ is $\mu$-hyperfinite, $G$ admits a $\mu$-measurable $\Pi$-coloring.
    \end{enumerate}
\end{definition}

Thus, by Theorems \ref{thm:ASItoasi} and \ref{thm:asiresults}, $\textnormal{ASI}(1) \subset \textnormal{Baire}$ and $\textnormal{ASI}(1) \subset \textnormal{HypMeas}$. These inclusions, and the one from Theorem \ref{thm:ASItocomp}, cannot all be equalities since the classes Baire, HypMeas, and HComp do not all coincide, even for computable locally checkable labeling problems. We go over some examples of differences between them now:

First, let $\Gamma$ be the class of all graphs, and $\Pi$ the problem of labeling vertices with natural numbers so that no label is repeated on any connected component. Clearly $(\Pi,\Gamma)$ is not in Baire or HypMeas, as a Borel graph admitting a Borel $\Pi$-coloring must have a smooth connectedness equivalence relation. On the other hand, $(\Pi,\Gamma)$ is clearly in HComp, as there are computable bijections between $\HF$ and $\omega$. 

Note that the previoius problem is not locally checkable. We would like to ask whether there are any examples as above with $\Pi$ locally checkable. Unfortunately, the lack of any restriction on $\Gamma$ still allows for vacuous counterexamples. We do not know a natural but not overly restrictive restriction on $\Gamma$, so we simply ask our question for several concrete cases:

\begin{question}\label{q:comp_to_baire}
    Let $(\Pi,\Gamma)$ be a locally checkable coloring problem with $\Gamma$ either the class of all graphs, the class of acyclic graphs, or the class of degree regular graphs. If $(\Pi,\Gamma) \in \textnormal{HComp}$, is it in Baire? is it in HypMeas?
\end{question}


In the other direction, we can show that $\textnormal{Baire} \cap \textnormal{HypMeas}$ is not contained in $\textnormal{HComp}$ even for locally checkable problems. Let $\Gamma$ be the class of bipartite graphs equipped with a binary predicate, say $\triangleleft$, whose restriction to each connected component is a well order. Let $\Pi$ be the problem of proper 2-coloring. 

First, $(\Pi,\Gamma) \in \textnormal{Baire} \cap \textnormal{HypMeas}$ because, in fact, every Borel $G \in \Gamma$ admits a Borel proper 2-coloring. This can be obtained by setting the color of each vertex to the mod 2 value of its distance to the $\triangleleft$-minimal element of its connected component.

On the other hand we can see $(\Pi,\Gamma) \notin \textnormal{HComp}$. This is because if it were, then so would be the related problem where we are not given $\triangleleft$, since we can always produce such a relation in the computable setting ourselves. That is, it would be the case that every bipartite highly computable graph admits a computable proper 2-coloring. Schmerl \cite{schmerl_1980} has constructed a counterexample to this, though.

This example may seem cheap, as the class $\Gamma$ has a very non-local definition. For a less cheap but much less trivial example, fix $k \in \omega$ and let $\Gamma$ be the class of $k$-regular bipartite graphs and $\Pi$ the problem of producing a perfect matching. Manaster and Rosenstein \cite{MR_matching} have produced counterexamples to show $(\Pi,\Gamma) \not\in \textnormal{HComp}$ for each $k \geq 2$. On the other hand, Matt Bowen (personal communication) has recently shown $(\Pi,\Gamma) \in \textnormal{Baire} \cap \textnormal{HypMeas}$ for $k$ odd.

We end this list of examples by mentioning a differnce between Baire and HypMeas. Let $\Gamma$ be the class of acyclic graphs, and $\Pi$ the problem of selecting finitely many ends from each connected component. See, for example, \cite{HM_endsII} for a precise definition. It is shown in \cite{HM_endsII} that this problem is not in Baire, but in \cite{jkl} that it is in HypMeas. This problem is not locally checkable, but in upcoming work of the second author and Bernshteyn, we will show that a certain computable locally checkable problem is equivalent to it, establishing that these classes differ even for such problems.

One question we can ask, however, is whether when these classes do coincide, there is an explanation via ASI algorithms:

\begin{question}\label{q:unification}
    Let $(\Pi,\Gamma)$ be a computable locally checkable coloring problem. If $(\Pi,\Gamma) \in \textnormal{Baire} \cap \textnormal{HypMeas} \cap \textnormal{HComp}$, is it in $\textnormal{ASI}_c(1)$?
\end{question}

Actually, we can already provide a negative answer to this question: Let $\Gamma$ be the class of acyclic graphs of minimum degree $\geq 2$ which admit no infinite injective paths on which every other vertex has degree 2. Let $\Pi$ be the problem of perfect matching. In \cite{conley.miller.pm}, it is shown that $(\Pi,\Gamma) \in \textnormal{Baire} \cap \textnormal{HypMeas}$. Ideas from that paper can be used to see it is also in HComp:

\begin{prop}\label{prop:no_bad_paths}
$(\Pi,\Gamma)$ as in the previous paragraph is in $\textnormal{HComp}$.
\end{prop}

All we need is the following easy lemma, implicit in \cite{conley.miller.pm}:

\begin{lemma}\label{lem:no_bad_paths}
    Let $G \in \Gamma$ and $x \in V(G)$. There exists a finite matching, say $M \subset G$, such that if $E$ denotes the set of vertices covered by $M$, then
    \begin{enumerate}
        \item $x \in E$.
        \item $G \res (V(G) \setminus E)$ still has minimum degree $\geq 2$.
    \end{enumerate}
\end{lemma}

\begin{proof}
    Let $M_0 = \{e\}$ for any edge $e$ with $x$ as an endpoint. We will define matchings $M_0 \subset M_1 \subset M_2 \subset \cdots$ inductively so that if $E_n$ is the set of vertices covered by $M_n$, each $E_n$ is $G$-connected: 
    
    Given $M_n$, consider the set $A_n$ of $y \in V(G) \setminus E_n$ such that the degree of $y$ in $G \res (V(G) \setminus E_n)$ is $\leq 1$. Note that since $G$ is acyclic, $E_n$ is $G$-connected, and $G$ has minimum degree $\geq 2$, each $y \in A_n$ must have degree 2 in $G$, and degree 1 in our restriction. By definition of degree, then, each such $y$ has a unique neighbor not in $E_n$, call it $y'$. Define $M_{n+1}$ by adding $(y,y')$ to $M_n$ for each $y \in A_n$. Note that $M_{n+1}$ is still a matching since $y \mapsto y'$ is injective, since $G$ is acyclic and $E_n$ is $G$-connected. $E_{n+1}$ is still $G$-connected since each $y \in A_n$ must have had a $G$-edge to $E_n$.
    
    We claim this construction must stop at some finite stage. That is, there must be some $n$ with $M_n = M_{n+1}$. If not, since $G$ is locally finite, by K\"{o}nig's Lemma we get an infinite injective path $y_0,y_0',y_1,y_1',\ldots$ with each $y_n \in A_n$. But as we noted in the previous paragraph, each $y_n$ then has $G$-degree 2, so no such path can exist by definition of $\Gamma$.
    
    This means we have an $n$ with $A_n = \emptyset$. We let $M = M_n$. (1) holds by definition of $M_0$. By definition of $A_n$, (2) holds as well.
\end{proof}

We now prove Proposition \ref{prop:no_bad_paths}:

\begin{proof}
    Let $G \in \Gamma$ be highly computable. Assume WLOG $V(G) \subset \omega$. Let $G_0 = G$.
    
    Let $n \in \omega$ and suppose inductively that we have computed some $G_n$ an induced subgraph of $G$ with $G_n \in \Gamma$. If $n \notin V(G_n)$, set $M_n = \emptyset$ and $G_{n+1} = G_n$. Else, choose a finite matching $M_n$ as in Lemma \ref{lem:no_bad_paths} which covers $n$, and let $G_{n+1}$ be the subgraph induced by $G_n$ on the set of uncovered vertices. We can find such an $M_n$ computably since $G$ is highly computable and the conditions from the lemma only require checking the degrees of vertices neighboring our finite set of vertices covered by $M_n$. Note that condition (2) from the lemma gives us our needed inductive hypothesis $G_{n+1} \in \Gamma$. Indeed, the minimum degree condition in (2) is one of the criteria for membership in (2), and the other criterion is automatic as it concerns infinite injective paths and we have removed only finitely many vertices.
    
    In the end, we have a computable matching $M := \bigcup_n M_n$. Since we ensured that the vertex $n$ was covered at stage $n$ in the construction, $M$ is a perfect matching.
\end{proof}

We now construct an example to show that this problem cannot be solved by an ASI algorithm. In fact, we have:

\begin{prop}
 $(\Pi,\Gamma)$ as in the previous proposition is not in $\textnormal{asi}(1)$.
\end{prop}

\begin{proof}
    We will construct a Borel graph $G \in \Gamma$ with $\asi(G) = 1$ but so that $G$ does not admit a Borel perfect matching. First, by \cite{hyperfinite.acyclic}, there is a hyperfinite acyclic 3-regular Borel graph, say $H$, which does not admit a Borel perfect matching.
    
    Let $E_1 \subset E_2 \subset \cdots$ be a sequence of finite Borel equivalence relations which union to $E_H$. Define $G$ by replacing each edge $(x,y) \in H$ with a path $x,z^{x,y}_1,\ldots,z^{x,y}_{4n},y$, where each $z_i^{x,y}$ is a new (degree 2) vertex and $n$ is minimal such that $x E_n y$. Call this value $\tau(x,y)$ for later use.
    
    First note that since every edge in $H$ has been replaced with a path of odd length, any perfect matching of $G$ yields a perfect matching of $H$ in a manner which preserves Borel-ness, and so $G$ does not admit a perfect matching.
    
    Also note $G \in \Gamma$: It is still acyclic, and its minimum degree is 2. The criterion regarding infinite paths with every other vertex having degree 2 is easy to check since, again, our added paths have odd length, and the vertices at their endpoints still have degree 3 in $G$.
    
    We now show $\asi(G) = 1$: Let $r \in \omega$. Let $U_0 = B_G(V(H),r)$, and $U_1 = V(G) \setminus U_0$. These sets are clearly Borel. We first show that $G^r \res U_0$ is component finite. It suffices to show that each $G^r \res U_0$-component is contained in $B_G(C,r)$ for some $E_r$-class $C$. Indeed, suppose $z,z' \in U_0$ with $\rho_G(z,z') \leq r$, and let $x,x' \in V(H)$ with $\rho_G(z,x),\rho_G(z',x') \leq r$ (so these witness $z,z' \in U_0$). Then $\rho_G(x,x') \leq 3r$. If $x = y_0,\ldots,y_n = x'$ is the unique $H$-path from $x$ to $x'$, then the unique $G$-path from $x$ to $x'$ must pass through each $y_i$ in turn. It follows from the definition of $G$ that $\tau(y_i,y_{i+1}) < r$ for each $i$, and so certainly $y_i E_r y_{i+1}$ for each $i$, so since $E_r$ is transitive we are done.
    
    It remains to show that $G^r \res U_1$ is component finite. It is clear, though, that the connected components of this graph are exactly the sets $$\{z_{r+1}^{x,y},\ldots,z_{4\tau(x,y)-r}^{x,y}\}$$ for $(x,y) \in H$ with $\tau(x,y) > r/2$ (just large enough to make the set nonempty).
\end{proof}

It is unclear whether, morally, this example indicates an actual failure of ASI algorithms to provide unification, or simply that the original definition of an ASI witness is not quite the correct one. After all, it always treats the parameter $r \in \omega$ (the scale) as a global constant, but the example $G$ above took advantage of the fact that for each $r$, $G^r$ looked very different around different vertices. 

As an alternate and potentially more convincing example, then, let us consider the problem of ``unbalanced matching'': Let $d \in \omega$, and Let $\Gamma$ be the class of bipartite graphs $G$, say with bipartition $V(G) = A \sqcup B$, such that for all $x \in A$, $\deg_G(x) > d$, while for all $x \in B$, $\deg_G(x) \leq d$. Let $\Pi$ be the problem of finding a matching of such $G$ which covers each vertex in $A$.

Hall's condition easily implies that, classically, any such $G$ admits such a matching. Kierstead \cite{kierstead_hall} gave a strengthening of Hall's condition still satisfied by graphs in $\Gamma$ and strong enough to imply the existence of computable matchings for highly computable graphs, hence $(\Pi,\Gamma) \in \textnormal{HComp}$. Marks and Unger \cite{MU_matching} did the same for the Baire measurable setting. The two strengthenings of Hall's condition are similar; where Hall's condition asks that a finite set always has at least as many neighbors as its cardinality these conditions ask that, moreover, if the set is large enough, it has many more neighbors than its cardinality. (The exact quantification of this is where the conditions differ.) 

Furthermore, the two proofs in these papers that these conditions imply the existence of definable matchings are extremely similar; both build the matching incrementally, at each stage picking edges out of things which look like the desired matchings locally in a way which preserves the strengthened Hall's condition. However, this common argument style does not have the appearance of an ASI algorithm, and the following is open:

\begin{question}
    Is $(\Pi,\Gamma)$ from the previous two paragraphs in $\textnormal{ASI}(1)$?
\end{question}

We also mention that, more recently, Bowen has shown (personal communication) $(\Pi,\Gamma) \in \textnormal{MeasHyp}$ by a somewhat different argument.

\subsection{ASI vs TOAST}\label{subsec:toast}

We now elaborate on the connection mentioned in Section \ref{sec:intro} between our ASI algorithms and TOAST algorithms from \cite{GR_grids}. The name ``TOAST algorithm'' comes from the following notion from \cite{GJKStoast}.

\begin{definition}
    Let $G$ be a Borel graph, and $r \in \omega$. A \textnormal{Borel $r$-toast} for $G$ is a Borel set $\mathcal{C} \subset [V(G)]^{<\omega}$ such that
    \begin{enumerate}
        \item Each $C \in \C$ is $G$-connected.
        \item For all $x,y \in V(G)$ in the same $G$-component, there is some $C \in \C$ with $x,y \in C$.
        \item For $C \neq D \in \C$, either $\rho_G(C,D) \geq r$, or $B_G(C,r) \subset D$.
    \end{enumerate}
\end{definition}

It is easy to see that if $G$ admits a Borel $r$-toast for all $r \in \omega$, then $\asi(G) \leq 1$. (ASI-witnesses can be found by looking at points near the boundaries of elements of toasts.) The converse seems open.

\begin{question}
    If a Borel graph $G$ has $\asi(G) \leq 1$, does it admit a Borel $r$-toast for every $r \in \omega$?
\end{question}

Just as our notion of an ASI algorithm was designed to capture the Borel combinatorics of Borel graphs with finite asi, the notion of a TOAST algorithm is designed to capture the Borel combinatorics of Borel graphs which admit Borel toasts. The connection between ASI and TOAST algorithms, however, seems less clear cut. Letting Toast denote the class of coloring problems solvable by a TOAST algorithm, we ask:

\begin{question}
    Are $\textnormal{ASI}(1)$ and Toast equal? Is one contained in the other?
\end{question}

\section*{Acknowledgements} We thank Anton Bernshteyn and Matt Bowen for encouraging feedback and for piquing our interest in the unbalanced matching problem. We thank Clinton Conley for many helpful conversations which contributed to the eventual content of this paper. The second author was partially supported by the ARCS foundation, Pittsburgh chapter.
\bibliographystyle{amsalpha} 
\bibliography{main}

\newcommand{\etalchar}[1]{$^{#1}$}
\providecommand{\bysame}{\leavevmode\hbox to3em{\hrulefill}\thinspace}
\providecommand{\MR}{\relax\ifhmode\unskip\space\fi MR }
\providecommand{\MRhref}[2]{%
  \href{http://www.ams.org/mathscinet-getitem?mr=#1}{#2}
}
\providecommand{\href}[2]{#2}
\begin{thebibliography}{CJM{\etalchar{+}}20b}

\bibitem[BCG{\etalchar{+}}21]{G+_trees}
Sebastian Brandt, Yi-Jun Chang, Jan Grebík, Christoph Grunau, Václav Rozhoň,
  and Zoltán Vidnyánszky, \emph{Local problems on trees from the perspectives
  of distributed algorithms, finitary factors, and descriptive combinatorics},
  2021, preprint, arXiv:2106.02066.

\bibitem[Bea76]{bean}
Dwight~R. Bean, \emph{Effective coloration}, The Journal of Symbolic Logic
  \textbf{41} (1976), no.~2, 469--480.

\bibitem[Ber20]{Bernshteyn.distributed}
Anton Bernshteyn, \emph{Distributed algorithms, the lov\'{a}sz local lemma, and
  descriptive combinatorics}, preprint, arXiv:2004.04905.

\bibitem[BW21]{BW.Konig}
Matt Bowen and Felix Weilacher, \emph{Definable k\"{o}nig theorems}, 2021,
  preprint, arXiv:2112.10222.

\bibitem[CHL{\etalchar{+}}19]{distributed_edge}
Yi-Jun Chang, Qizheng He, Wenzheng Li, Seth Pettie, and Jara Uitto,
  \emph{Distributed edge coloring and a special case of the constructive
  lov\'{a}sz local lemma}, ACM Trans. Algorithms \textbf{16} (2019), no.~1.

\bibitem[CJM{\etalchar{+}}20a]{dimension}
Clinton Conley, Steve Jackson, Andrew Marks, Brandon Seward, and Robin
  Tucker-Drob, \emph{Borel asymptotic dimension and hyperfinite equivalence
  relations}, 2020, preprint, arXiv:2009.06721.

\bibitem[CJM{\etalchar{+}}20b]{hyperfinite.acyclic}
Clinton Conley, Steve Jackson, Andrew Marks, Brandon Seward, and Robin
  Tucker-Drob, \emph{Hyperfiniteness and {B}orel combinatorics}, J. Eur. Math.
  Soc. (JEMS) \textbf{22} (2020), no.~3, 877--892.

\bibitem[CM16]{conley.miller.toast}
Clinton~T. Conley and Benjamin~D. Miller, \emph{A bound on measurable chromatic
  numbers of locally finite {B}orel graphs}, Math. Res. Lett. \textbf{23}
  (2016), no.~6, 1633--1644.

\bibitem[CM17]{conley.miller.pm}
\bysame, \emph{Measurable perfect matchings for acyclic locally countable
  {B}orel graphs}, J. Symb. Log. \textbf{82} (2017), no.~1, 258--271.

\bibitem[EFK84]{VAL}
A.~Ehrenfeucht, V.~Faber, and H.A. Kierstead, \emph{A new method of proving
  theorems on chromatic index}, Discrete Mathematics \textbf{52} (1984), no.~2,
  159--164.

\bibitem[GJKS]{GJKStoast}
S.~Gao, S.~Jackson, E.~Krohne, and B~Seward, \emph{Forcing constructions and
  countable borel equivalence relations}, preprint, arxiv:1503.07822.

\bibitem[GR21]{GR_grids}
Jan Grebík and Václav Rozhoň, \emph{Local problems on grids from the
  perspective of distributed algorithms, finitary factors, and descriptive
  combinatorics}, 2021, preprint, arXiv:2103.08394.

\bibitem[HM09]{HM_endsII}
Greg Hjorth and Benjamin Miller, \emph{Ends of graphed equivalence relations,
  ii}, Israel Journal of Mathematics \textbf{169} (2009), 393--415.

\bibitem[JKL02]{jkl}
S.~Jackson, A.~S. Kechris, and A.~Louveau, \emph{Countable {B}orel equivalence
  relations}, J. Math. Log. \textbf{2} (2002), no.~1, 1--80. \MR{1900547}

\bibitem[Kie81]{nearly_perfect}
Henry~A. Kierstead, \emph{Recursive colorings of highly recursive graphs},
  Canadian Journal of Mathematics \textbf{33} (1981), no.~6, 1279–1290.

\bibitem[Kie83]{kierstead_hall}
Henry~A. Kierstead, \emph{An effective version of hall's theorem}, Proceedings
  of the American Mathematical Society \textbf{88} (1983), no.~1, 124--128.

\bibitem[Mar16]{Marks}
Andrew~S. Marks, \emph{A determinacy approach to {B}orel combinatorics}, J.
  Amer. Math. Soc. \textbf{29} (2016), no.~2, 579--600. \MR{3454384}

\bibitem[Mil09]{miller.2end}
Benjamin Miller, \emph{Ends of graphed equivalence relations, i}, Israel
  Journal of Mathematics \textbf{169} (2009), 375--392.

\bibitem[MR72]{MR_matching}
Alfred~B. Manaster and Joseph~G. Rosenstein, \emph{{Effective Matchmaking
  (Recursion Theoretic Aspects of a Theorem of Philip Hall)}}, Proceedings of
  the London Mathematical Society \textbf{s3-25} (1972), no.~4, 615--654.

\bibitem[MU15]{MU_matching}
Andrew~S. Marks and Spencer Unger, \emph{Baire measurable paradoxical
  decompositions via matchings}, Advances in Mathematics \textbf{289} (2015),
  397--410.

\bibitem[NS93]{NS_lcl}
Moni Naor and Larry Stockmeyer, \emph{What can be computed locally?}, SIAM J.
  Comput, 1993, pp.~184--193.

\bibitem[Sch80]{schmerl_1980}
James~H. Schmerl, \emph{Recursive colorings of graphs}, Canadian Journal of
  Mathematics \textbf{32} (1980), no.~4, 821–830.

\bibitem[Wei22]{weilacher_2chi-1}
Felix Weilacher, \emph{Descriptive chromatic numbers of locally finite and
  everywhere two-ended graphs}, Groups Geom. Dyn. (2022), no.~1, 141--152.

\end{thebibliography}

\end{document}